\documentclass[10pt]{amsart}
\usepackage{amssymb}
\usepackage{amsmath}
\usepackage{amsthm}
\usepackage{latexsym}
\usepackage{graphicx}
\usepackage{hyperref}

\newcommand{\Z}{{\mathbb Z}}
\newcommand{\R}{{\mathbb R}}
\newcommand{\C}{{\mathbb C}}

\newcommand{\T}{{\mathbb T}}

\newtheorem{lemma}{Lemma}[section]
\newtheorem{theorem}[lemma]{Theorem}

\newtheorem{proposition}[lemma]{Proposition}

\newtheorem{definition}[lemma]{Definition}

\newcommand{\be}{\begin{equation}}
\newcommand{\ee}{\end{equation}}

\newcommand{\ol}{\overline}
\newcommand{\ti}{\tilde}

\newcommand{\spr}[2]{\left\langle #1 , #2 \right\rangle}

\newcommand{\E}{\mathrm{e}}
\newcommand{\I}{\mathrm{i}}

\newcommand{\tr}{\mathrm{tr}}

\DeclareMathOperator{\dist}{dist}

\newcommand{\eps}{\varepsilon}

\numberwithin{equation}{section}

\begin{document}

\title[AC Spectrum for quasi-periodic operators]{Absolutely continuous spectrum for quasi-periodic Schr\"odinger operators}

\author[H.\ Kr\"uger]{Helge Kr\"uger}
\address{Mathematics 253-37, Caltech, Pasadena, CA 91125}
\email{\href{helge@caltech.edu}{helge@caltech.edu}}
\urladdr{\href{http://www.its.caltech.edu/~helge/}{http://www.its.caltech.edu/~helge/}}

\thanks{H.\ K.\ was supported by the Simons Foundation.}

\date{\today}

\keywords{absolutely continuous spectrum, quasi-periodic
 Schr\"odinger operators}

\begin{abstract}
 I prove that quasi-periodic Schr\"odinger operators in arbitrary dimension
 have some absolutely continuous spectrum. 
\end{abstract}

\maketitle

%
%
%

\section{Introduction}

Let $d\geq 1$ and consider the family of Schr\"odinger operators
$H_{\lambda,\alpha,x}=\Delta +\lambda V_{\alpha,x}$ acting on
$\ell^2(\Z^d)$ where $\lambda>0$ is a coupling constant,
\be
 \Delta\psi(n)=\sum_{|e|_1=1}\psi(n+e),\quad |x|_1=|x_1|+\dots+|x_d|
\ee
is the discrete Laplacian, and the potential $V_{\alpha,x}$ is
the multiplication operator with the sequence
\be
 V_{\alpha,x}(n)= f(x+\alpha\star n),\quad 
 (\alpha\star n)_j = \alpha_j n_j
\ee
$\alpha\in \R^d$, $x\in\T^d$, $\T=\R/\Z$, and $f:\T^d\to\R$
is a non constant real-analytic function. My main goal will
be to show

\begin{theorem}\label{thm:main1}
 Let $\eps>0$. Then there exists $\lambda_0>0$ such that
 for $\lambda\in (0,\lambda_0)$, there exists a set of frequencies
 $\mathcal{A}_{\lambda}\subseteq [0,1]^d$ of measure
 $|\mathcal{A}_{\lambda}|\geq 1-\eps$ such that for $\alpha\in\mathcal{A}_{\lambda}$,
 $x\in\T^d$ the Schr\"odinger operator $H_{\lambda,\alpha,x}=\Delta+\lambda
 V_{\alpha,x}$ has some absolutely continuous spectrum.
\end{theorem}

Previously, Bourgain \cite{bbook}, \cite{b07} has shown that this class of operators
exhibits extended states and in particular the spectrum
is purely continuous. Denoting by $\sigma_{\mathrm{ac}}(H_{\lambda,\alpha,x})$
the absolutely continuous spectrum of $H_{\lambda,\alpha,x}$,
I will in fact show that
\be
|\sigma(\Delta)\setminus\sigma_{\mathrm{ac}}(H_{\lambda,\alpha,x})| \leq \eps,
  \quad\sigma(\Delta)=[-2d,2d].
\ee
Unfortunately, I am unable to address the structure of 
$\sigma_{\mathrm{ac}}(H_{\lambda,\alpha,x})$ as a set or to
show that the absolutely continuous spectrum is pure.

I want to mention at this point that the results of this paper
can be extended in various ways with minimal effort. Maybe, the
most important one is that instead of considering $\psi$ to be
real valued, one could take $\psi(n)\in\C^k$ and $V(n)$ to be
an appropriate Hermitian matrix. This extension would allow one
to first consider Schr\"odinger operators on so called
{\em strips}, i.e. defining the Laplacian on $\Z\times\{1,\dots,W\}$
for some $W\geq 2$. Furthermore, this would allow one to consider
periodic directions in the definition of $V$. Finally, it
should be possible to replace the Laplacian $\Delta$ by
a more general hopping operator $T$ of the form
\be
 T\psi(n)=\sum_{k\in\Z^d} t_k \psi(n+k)
\ee
for $|t_k|\leq \E^{-\eta |k|}$ for some $\eta > 0$.

\bigskip

The proof to exhibit absolutely continuous spectrum proceeds
by first studying the dual operator $\widehat{H}_{\lambda,\alpha,x}$
which exhibits Anderson localization. 
As already mentioned, this implies that the
operator $H_{\lambda,\alpha,x}$ has extended states. In
order to obtain absolutely continuous spectrum, I will
need to obtain further control on the eigenvalues of
$\widehat{H}_{\lambda,\alpha,x}$. This will be done by using
the methods of \cite{kskew}. For an implentation of this strategy
in the simpler context of limit-periodic Schr\"odinger operators
see \cite{klp}. For further consequences of the eigenvalue
perturbation results for the skew-shift, see \cite{kcon}.
I expect that the methods of \cite{kcon} can be used to prove
results about the integrated density of states and eigenvalue
statistics for quasi-periodic Schr\"odinger operators. However,
the statements will be weaker as the set of eliminated frequencies
depends on the energy.

\bigskip

Let us now review what else is known about the absolutely continuous
spectrum for Schr\"odinger operators. Proving absolutely continuous
spectrum for the free and periodic Schr\"odinger operators can
be done by fairly standard methods. Proofs in dimension one usually
rely on the fact that the Schr\"odinger equation
\be
 H_x u(n)=u(n+1)+u(n-1)+ f(T^n x) u(n)
\ee
is equivalent to the transfer matrix equation
\be
 \begin{pmatrix} u(n+1)\\ u(n)\end{pmatrix}= A(T^n x)  
  \begin{pmatrix} u(n)\\ u(n-1)\end{pmatrix},
   \quad A(x)=\begin{pmatrix} E-f(x)&-1\\1&0\end{pmatrix},
\ee
where $Tx=x+\alpha\pmod{1}$. Using methods based on KAM, one
can show absolutely continuous spectrum by showing that
the cocycle
\be
 \T\times\C^2\ni(x,v)\mapsto(x+\alpha,A(x) v)
\ee
is conjugated to a cocycle of rotations. This
was first done by Dinaburg and Sinai \cite{ds75} then improved by
Eliasson \cite{e92} to prove purely absolutely continuous spectrum.
The final improvement was by Avila and Jitomirskaya \cite{aj10}
using input based on duality. Furthermore, Kotani theory
\cite{d07} allows one to describe the absolutely continuous
spectrum by the vanishing of the Lyapunov exponent,
$L(E)=\lim_{n\to\infty}\frac{1}{n}\int\log\|A(T^{n-1}x)\cdots
A(x)\|dx$. This has allowed Avila \cite{apre} to give an
excellent description of the absolutely continuous spectrum.

For operators in higher dimensions, the main work on
understanding almost-periodic potentials has been done in
the continuum. In \cite{kl1}, Karpeshina and Lee have
shown that the Schr\"odinger operator $H=-\Delta+V$ acting
on $L^2(\R^2)$ with $V$ a small enough limit-periodic potential
has some absolutely continuous spectrum
at high energies and that the spectrum contains a semi-axis.
In the discrete case, I have shown \cite{klp} that
limit-periodic potentials which are sufficiently well approximated
by periodic potentials lead to purely absolutely continuous
spectrum. 

Finally, in dimension two and for polyharmonic operators,
i.e. $(-\Delta)^{\ell}$ for $\ell\geq 2$, Karpeshina
and Shterenberg have exhibited in \cite{ks} the existence
of an absolutely continuous component in the spectrum
of a quasi-periodic operator on $L^2(\R^2)$.

\bigskip

Our conceptual understanding of the absolutely continuous
spectrum in dimensions $d\geq 2$ is much weaker than in
one dimension. Due to subordinacy theory \cite{jl99}, \cite{gp87},
we can morally characterize the absolutely continuous spectrum
as the set of energies, such that there exists a bounded
solution. The best, we can do in higher dimensions is
\cite{kl00}, which is too weak to imply the results of this
paper or \cite{klp}.

In the case of ergodic operators so in particular
almost-periodic ones, one even has Kotani theory and
thus can describe the absolutely continuous spectrum
as the essential closure of the set of energies where
the Lyapunov exponent vanishes.

Finally, let me mention that proving pure point spectrum
respectively proving that the spectrum is purely continuous
is a statement about solutions as one shows that all generalized
eigenfunctions are either square integrable or not. It would
be very interesting to obtain a description of the absolutely
continuous spectrum in terms of solutions in arbitrary dimension.

\bigskip

My methods are based on considering the dual operator.
This concept has been introduced in \cite{a78}, \cite{aa80}
for the almost Mathieu operator. This concept was further
developed in \cite{gjls97}, \cite{bj02}.

As already mentioned the proof in this paper proceeds
by using that Anderson localization holds for the dual
operator $\widehat{H}_{\lambda,\alpha}$ defined in
\eqref{eq:defdual1} respectively its fibers see
\eqref{eq:defdual2}. This implies in particular that
for the almost every energy $E$ in the absolutely continuous
component exhibited in Theorem~\ref{thm:main1} there
exists an extended state. By an extended state,
I mean in this context an almost-periodic solution 
$u:\Z^d\to\C$ of the eigenvalue problem
\be
 H_{\lambda,\alpha,x} u = E u.
\ee
The main problem in order to carry out such a construction
is that one needs to relate the eigenfunctions of
$\widehat{H}_{\lambda,\alpha,x}^{\Lambda_{r}(0)}$
and $\widehat{H}_{\lambda,\alpha,x}^{\Lambda_{R}(0)}$
for $r < R$, where $\Lambda_r(0) = [-r,r]^d$. This is
only possible if one knows that the corresponding
eigenvalues are simple. In order to ensure this, we
use the methods of \cite{kskew}, which put as into
a perturbative regime. This problem was solved
in \cite{klp} by a novel estimate.

I believe that obtaining an understanding of how to
prove simplicity of the eigenvalues in a problem of this
type, would lead to major improvements in all known
results.

%
%
%

\section{Strategy of the proof}

Let us now deal with the specifics of the proof. First, we can
write $f$ in Fourier series as
\be
 f(x) = \sum_{k\in\Z^d} \hat{f}(k) e(k\cdot x),\quad e(t)=\E^{2\pi\I t},
  \quad k\cdot x=\sum_{j=1}^{d} k_j x_j.
\ee
We will restrict ourself for simplicity to the potential
$V_{\alpha}(n)=f(n\star\alpha)$. For $u\in\ell^{1}(\Z^d)$,
we define its Fourier transform by $\hat{u}(x)=\sum_{k\in\Z^d}
e(k\cdot x) u(k)$. The Fourier transform of
$(\Delta + \lambda V_{\alpha})\psi$ is given by
\be
 \sum_{j=1}^{d} 2\cos(2\pi x_j) \hat{\psi}(x)+\lambda
  \sum_{k\in\Z^{d}}\hat{f}(k)\hat\psi(x+k\star\alpha).
\ee
We define the operator $\widehat{H}_{\lambda,\alpha}: L^2(\T^d)\to L^2(\T^d)$ by
\be\label{eq:defdual1}
 \widehat{H}_{\lambda,\alpha} \psi(x) =
  \sum_{j=1}^{d} 2\cos(2\pi x_j)\psi(x)+\lambda
   \sum_{k\in\Z^{d}}\hat{f}(k)\psi(x+k\star\alpha).
\ee
We see that $\widehat{H}_{\lambda,\alpha}\psi(x)$ only depends on
$\{\psi(x+k\star\alpha)\}_{k\in\Z^d}$. It thus makes
sense to consider the fibered  operator
\be\label{eq:defdual2}
 \widehat{H}_{\lambda,\alpha,x} \psi(n)=
  \left(\sum_{j=1}^{d} 2\cos(2\pi (x_j +n_j\alpha_j))
   \right)\psi(n)+\lambda\cdot \sum_{k\in\Z^d} \widehat{f}(k) \psi(n+k).
\ee
As $f$ is real-analytic, we have that $|\hat{f}(k)|\leq C \E^{-\eta |k|}$
for $C,\eta > 0$. For simplicity, we will assume in the following that
$|\hat{f}(k)|\leq\E^{-\eta |k|}$, which is possible by changing $\lambda>0$.
We can write $\widehat{H}_{\lambda,\alpha,x}=\lambda T + W_{\alpha,x}$,
where $T$ is the hopping operator
\be
T\psi(n)=\sum_{k\in\Z^d} \widehat{f}(k) \psi(n+k)
\ee
and $W_{\alpha,x}$ is the multiplication operator by the
sequence 
\be
 W_{\alpha,x}(n)=\sum_{j=1}^{d} 2\cos(2\pi(x_j+n_j\alpha_j))=W(x+n\star\alpha),
  \quad W(x)=\sum_{j=1}^{d}2\cos(2\pi x_j).
\ee
In particular, $\widehat{H}_{\lambda,\alpha,x}$ is again
a quasi-periodic operator but the coupling constant is changed
from small to large.
Finally, we note that we view $\widehat{H}_{\lambda,\alpha,x}$
as an operator acting on $\ell^2(\Z^d)$. 

Given $\Lambda\subseteq\Z^d$, we denote by $A^{\Lambda}$ the
restriction to $\ell^2(\Lambda)$ of an operator
$A:\ell^2(\Z^d)\to\ell^2(\Z^d)$. We introduce the
cube
\be
 \Lambda_r(n)=\{x\in\Z^d:\quad |n-x|_{\infty}\leq r\}
\ee
where $|x|_{\infty}=\max(|x_1|,\dots,|x_d|)$.  Finally,
we introduce the following definition.

\begin{definition}
 An eigenvalue $E$ of a self-adjoint operator $A$ is
 called $\delta$-simple if $\tr(P_{[E-\delta,E+\delta]}(A))=1$.
\end{definition}

The following theorem provides the perturbative analysis of eigenfunctions.
It is note worthy that the perturbation parameter is the
frequency $\alpha\in [0,1]^d$, which enters the problem
as the fast variable. 

\begin{theorem}\label{thm:stepA}
 Let $\eps\in (0,1)$ and $R_1 \geq 1$ be large enough.
 There exists $\lambda_1>0$ and sequences
 $R_1 < R_2 < R_3 < \dots$, $\delta_1 > \delta_2 > \delta_3 > \dots$
 such that for $\lambda\in (0,\lambda_1)$  and $y\in\T^d$, we have
 \begin{enumerate}
  \item $R_j = (R_{j-1})^{10}$, $\delta_1=\lambda^{\frac{1}{20}}$
   and $\delta_j = \lambda^{\frac{1}{20}} \exp( - (R_{j-1})^{\frac{1}{2}})$
   for $j\geq 2$.
  \item There exists $G_{y}\subseteq [0,1]^d$ of measure
   $|G_{y}|\geq 1- \eps$.
  \item For $j\geq 1$, there exists a function 
  $E_j:G_{y}\to\R$ such that for $\alpha\in G_{y}$
   \be
    E_j(\alpha)\in\sigma(
    \widehat{H}^{\Lambda_{R_j}(0)}_{\lambda,\alpha,y})
   \ee
   is $\delta_j$-simple.
  \item We have $|E_{j}(\alpha) - E_{j-1}(\alpha)|\leq (\delta_j)^{10}$ for $\alpha\in G_{y}$
   and $j\geq 0$. For $\psi_j$ the corresponding normalized eigenfunctions, we have
   \be
    \|\psi_j-\psi_{j-1}\|\leq (\delta_j)^3.
   \ee
 \end{enumerate}
\end{theorem}

In the theorem, when $j = 0$ we formally set $\psi_{-1}=\delta_0$ and
$E_{-1}=\sum_{j=1}^{d} 2\cos(2\pi y_j)$.
We will explain the proof of this theorem in Section~\ref{sec:pfstepA}.
Large parts of the proof follow ideas from \cite{kskew}.

In order to deduce Theorem~\ref{thm:main1}, we will need to
reformulate the conclusions of the previous theorem. In particular,
instead of fixing $x\in\T^d$ and varying $\alpha\in [0,1]^d$, we
will need to do the opposite. For this consider the set
\be
 A=\{(x,\alpha)\in\T^d\times [0,1]^d:\quad
  \alpha\in G_x\}.
\ee
By the previous theorem, we have that $|A|\geq 1-\eps$.
Introduce
\be
 \mathcal{A}=\{\alpha:\quad |\{x:\quad (x,\alpha)\in A\}|
  \geq 1-\eps^{\frac{1}{2}}\}.
\ee
One can check that $|\mathcal{A}|\geq 1-\eps^{\frac{1}{2}}$.
Let us now fix $\alpha\in\mathcal{A}$ and define
\be
 G=\{x:\quad (x,\alpha)\in A\}.
\ee
By construction, we clearly have that $|G|\geq 1-\sqrt{\eps}$.
For $x\in G$ the conclusions (ii)-(v) of the previous
theorem hold. Unfortunately, taking the limit of the eigenvalues,
we've constructed so far doesn't necessarily lead to a nice
function $E: G\to\R$. The following proposition remedies
this situation.

\begin{proposition}\label{prop:evfunc}
 Let $\eps_1 > \eps^{\frac{1}{2}}$. There exist $\delta>0$, $G_1\subseteq \T^d$,
 a function $\gamma:\T^d\to\R$, and a map $\psi:\T^d\to\ell^2(\Z^d)$
 such that
 \begin{enumerate}
  \item $|G_1\cap G| \geq 1-\eps_1$.
  \item $|\nabla\gamma(x)|\geq\delta$ for $x\in G_1$.
  \item For $x\in G_1\cap G$ we have 
   \be
    \gamma(x)\in \sigma(\widehat{H}_{\lambda,\alpha,x}).
   \ee
  \item $\|\psi(x)\|=1$ for $x\in G_1\cap G$ and $\psi(x)=0$
   for $x\in\T^d\setminus (G_1\cap G)$.
  \item $\widehat{H}_{\lambda,\alpha,x} \psi(x)=\gamma(x)\psi(x)$
   for $x\in G_1\cap G$.
  \item For $x\in G_1\cap G$, we have $\|\psi(x)\|_{\ell^1(\Z^d)} \leq 2$.
  \item Fix $x\in\T^d$ and let
   \be
    \mathcal{L}=\{\ell:\quad x+\ell\star\alpha\in G_1\cap G\},\quad
     \psi_\ell(x;n) = \psi(x-\ell\star \alpha; n + \ell).
   \ee
   Then the $\{\psi_{\ell}\}_{\ell\in\mathcal{L}}$ form an orthonormal set
   in $\ell^2(\Z^d)$ consisting of eigenfunctions of $\widehat{H}_{\lambda,\alpha,x}$.
   Finally $\widehat{H}_{\lambda,\alpha,x} \psi_{\ell}(x)=\gamma(x-\ell\star\alpha)\psi_{\ell}(x)$.
 \end{enumerate}
\end{proposition}

We denote by $\Gamma: L^2(\T^d)\to L^2(\T^d)$ the multiplication
operator by $\chi_{G\cap G_1} \gamma$. By (vi), we can define for $g\in L^{\infty}(\T^d)$
\be
 Q g(x)=\sum_{k\in\Z^d} q_{k}(x) g(x+k\star \alpha), \quad
  q_{k}(x)=\chi_{G}(x+k\star\alpha)\cdot\psi(x+k\star\alpha;-k).
\ee
A formal computation shows that $\widehat{H}_{\lambda,\alpha} Q = Q \Gamma$
is equivalent to
\be
 \sum_{j=1}^{d} 2\cos(2\pi x_j)q_{\ell}(x)+
  \lambda\sum_{k\in\Z^d} \hat{f}(k) q_{\ell-k}(x+k\star\alpha)
  =\gamma(x+\ell\star\alpha) q_{\ell}(x).
\ee
Using that $q_{k}(x)=\psi_{-k}(x;0)$ and $q_{\ell-k}(x+k\star\alpha)=\psi_{-\ell}(x;k)$
in the notation of Proposition~\ref{prop:evfunc} one easily verifies this.
The next lemma establishes that $Q$ is a bounded operator and thus that
we can make this computation.

\begin{lemma}
 The operator $Q$ is bounded $L^2(\T^d)\to L^2(\T^d)$.
\end{lemma}

\begin{proof}
 Let $g_G=\chi_G g$. Then we have that
 \[
  \|Qg\|_{L^{2}(\T^d)}^2 = \int_{\T^d}\sum_{n\in\Z^d} \ol{g_G(y)} g_{G}(y+n\star\alpha)
   \cdot \sum_{k\in\Z^d} \ol{\psi(y;-k)}\psi(y+n\star\alpha;-n-k)dy.
 \]
 By Proposition~\ref{prop:evfunc} (iv), we have that the sum over $k$
 is equal $1$ if $n=0$ and equal to $0$ otherwise. The claim follows.
\end{proof}

One can now formally compute the adjoint of $Q$ to be
\be
 Q^{\ast} g(x)=\sum_{k\in\Z^d} \chi_{G}(x)\ol{\psi(x;-k)}g(x-k\star\alpha).
\ee
A quick computation shows that $Q^{\ast} Q=Q Q^{\ast} = \chi_{G}$. In
particular, we have that $\|Q\|=1$.
We are now ready for

\begin{proof}[Proof of Theorem~\ref{thm:main1}]
 The previous lemma shows that $Q$ conjugates the restriction 
 $\widehat{H}_{\lambda,\alpha}^{G\cap G_1}$ to $L^2(G\cap G_1)$ to the
 multiplication operator by $\gamma$. Hence, it suffices
 to prove that for $E\in\R$ and $s>0$, we have that
 \[
 |\{x\in G:\quad \gamma(x)\in [E-s,E+s]\}|
   \lesssim \frac{s}{\delta}.
 \]
 This follows by the first part of Proposition~\ref{prop:evfunc}.
\end{proof}

%
%
%

\section{Niceness of the eigenvalue parametrisation:
 Proof of Proposition~\ref{prop:evfunc}}

In order to prove Proposition~\ref{prop:evfunc}, we will need
to reformulate the conclusions of Theorem~\ref{thm:stepA}
for fixed $\alpha\in\mathcal{A}$. There exists a sequence
of sets $G_j\subseteq\T^{d}$, functions $E_j: G_j\to\R$, and
$\psi_j: G_j\to\ell^2(\Lambda_{R_j}(0))$ with the following
properties
\begin{enumerate}
 \item $G_{j}\subseteq G_{j-1}$, $|G_{j}|\geq 1-\eps^{\frac{1}{2}}$.
 \item For $x\in G_j$, we have $E_j(x)\in\sigma(\widehat{H}_{\lambda,\alpha,x}
  ^{\Lambda_{R_j}(0)})$ is $\delta_j$ simple.
 \item For $x\in G_j$, $\widehat{H}_{\lambda,\alpha,x}^{\Lambda_{R_j}(0)}\psi_j(x)=E_j(x) \psi_j(x)$.
 \item For $x\in G_j$, $\|\psi_j(x)\|=1$.
 \item For $x\in G_{j}$, $|E_{j}(x)-E_{j-1}(x)|\leq (\delta_{j})^{10}$ 
  and $\|\psi_{j}(x)-\psi_{j-1}(x)\| \leq(\delta_j)^3$.
\end{enumerate}

We begin by understanding the limit of the functions $\psi_j(x)$.

\begin{lemma}
 For $x\in G=\bigcap_{j\geq 1} G_j$ with $|G|\geq 1-\eps^{\frac{1}{2}}$,
 there exists $\psi(x)\in\ell^2(\Z^d)$ such that $\psi$ solves
 $\widehat{H}_{\lambda,\alpha,x}\psi(x)=E(x)\psi(x)$ for some $E(x)$.
 We have that $\psi_j\to\psi$ in $\ell^2(\Z^d)$ and
 \be
  \|\psi(x)\|_{\ell^1(\Z^d)}\leq 2.
 \ee
\end{lemma}

\begin{proof}
 By (v) the $\psi_j(x)$ form a Cauchy sequence. Hence $\psi(x)$ exists. By
 continuity of the norm $\|\psi(x)\|=1$. Finally, $\psi(x)$ solves 
 the eigenvalue problem for $E(x)=\lim_{j\to\infty} E_j(x)$.

 As $\psi_{j}(x)$ is supported in a set containing less than $(3R_j)^d$
 many elements, we have
 \[
  \|\psi_{j}(x)-\psi_{j-1}(x)\|_{\ell^1(\Z^d)} \leq (3 R_j)^{\frac{d}{2}}
  (\delta_j)^3 \leq (\delta_j)^2.
 \]
 Hence, the $\psi_j(x)$ are also Cauchy in $\ell^1(\Z^d)$ and thus
 also is $\psi(x)$. Finally, we have
 \[
  \|\psi(x)\|_{\ell^1(\Z^d)} \leq 1 + \sum_{j=0}^{\infty} \|\psi_{j}(x)-\psi_{j-1}(x)\|_{\ell^1(\Z^d)}
   \leq 1 + 2\lambda^{\frac{1}{10}} \leq 2.
 \]
 as $\|\psi_{-1}\|_{\ell^1(\Z^d)}=1$.
\end{proof}

The next step in our analysis will be to understand the derivative
of the functions $E_{j}(x)$ for $x\in G_{j}$. The next lemma
implies in particular, that $\nabla E_{j}(x)$ makes sense.

\begin{lemma}\label{lem:evanal}
 Let $x\in G_{j}$ and $U=B_{(\delta_j)^2}(x)$. Then there exists
 an analytic function $f: U\to \R$ such that
 \begin{enumerate}
  \item $f(x)= E_j(x)$.
  \item For $y\in U$, $f(y)$ is a $\frac{1}{2}\delta_j$ simple
   eigenvalue of $\widehat{H}_{\lambda,\alpha,y}^{\Lambda_{R_j}(0)}$.
  \item For $y\in U$, $|f(y) - E_j(x)|\leq C |y-x|$ for some $j$
   independent $C>0$.
 \end{enumerate}
\end{lemma}

\begin{proof}
 For $y\in U$, we have
 \[
  \|\widehat{H}_{\lambda,\alpha,y}^{\Lambda_{R_j}(0)}-
   \widehat{H}_{\lambda,\alpha,x}^{\Lambda_{R_j}(0)}\| \leq \|\nabla W\|_{L^{\infty}(\T^d)} \delta_j^2
    \leq \frac{1}{4}\delta_{j},
 \]
 where $W(x)=\sum_{j=1}^{d} 2\cos(2\pi x_j)$.
 Define $f$ to be the analytic continuation of the necessarily
 simple eigenvalue $E_j(x)$. It is clear that (i) and (ii) hold.
 Furthermore, (iii) follows with $C=\|\nabla W\|_{L^{\infty}(\T^d)}$.
\end{proof}

\begin{lemma}
 There exists $\kappa > 0$ and a set $G^{\kappa}\subseteq\T^d$
 of measure $|G^{\kappa}|\geq 1-\eps^{\frac{1}{2}}$ such that
 for $x\in G^{\kappa}_j = G^{\kappa}\cap G_{j}$, we have
 \be
 |\nabla E_j(x)|
   \geq \kappa.
 \ee
\end{lemma}

\begin{proof}
 Recall that $\gamma_{-1}(x)=\sum_{j=1}^{d} 2\cos(2\pi x_j)$ and
 define
 \[
  G^{\kappa} = \{x\in\T^d:\quad |\nabla\gamma_{-1}(x)|\geq 2\kappa\}
 \]
 we have $|G^{\kappa}|\to 0$ as $\kappa\to 0$. So we may choose $\kappa$
 such that $|G^{\kappa}|=1 -\eps^{\frac{1}{2}}$.

 By the previous lemma, we can extend $E_j$ to an analytic function
 in a small neighborhood of $x$. Standard perturbation theory then
 implies
 \[
  \nabla E_j(x)=
   \spr{\psi_j(x)}{\nabla \widehat{H}_{\lambda,\alpha,x}^{\Lambda_{R_j}(0)} \psi_{j}(x)}.
 \]
 For $x\in G^{\kappa}$, we clearly have that
 \[
 |\spr{\delta_0}{\nabla \widehat{H}_{\lambda,\alpha,x}^{\Lambda_R(0)} \delta_0}|
   \geq 2 \kappa.
 \]
 We have that
 \[
  \|\psi_{j}(x)-\delta_0\|\leq\sum_{\ell=0}^{j}\|\psi_{j}(x)-\psi_{j-1}(x)\|
   \leq \lambda^{\frac{1}{10}}.
 \]
 Then as
 \[
 |\spr{\psi_j(x)}{\nabla \widehat{H}_{\lambda,\alpha,x}^{\Lambda_{R_j}(0)} \psi_{j}(x)}-
   \spr{\delta_0}{\nabla \widehat{H}_{\lambda,\alpha,x}^{\Lambda_R(0)} \delta_0}|
    \leq 2 \|\nabla W\|_{L^\infty(\T^d)} \|\psi_{j}(x)-\delta_0\|
 \]
 and $\lambda\leq \kappa^{10}$, the claim follows.
\end{proof}

We will now start to construct the function $\gamma$ described
in Proposition~\ref{prop:evfunc}. To do so, we will construct an
extension $\gamma_j : G^{\kappa}\to\R$ of $E_j: G^{\kappa}_j\to\R$.
This extension should have the properties
\begin{enumerate}
 \item $\gamma_j(x)=E_j(x)$ for $x\in G^{\kappa}_j$.
 \item $|\gamma_j(x)-\gamma_{j-1}(x)|\leq (\delta_j)^{10}$ for $x\in G^{\kappa}$.
 \item $|\nabla\gamma_j(x)|\geq\kappa - \delta_{1} - \dots - \delta_{j}$.
\end{enumerate}
These properties guarantee that (i) through (v) of Proposition~\ref{prop:evfunc}
hold as we have already observed that the eigenfunctions converge. Also (vi) holds,
we will prove (vii) at the end of this section. Let us now explain
how to construct $\gamma_j$. First, it is clear that the claim holds
for $\gamma_{-1}$, so we only have to construct $\gamma_j$ given $\gamma_{j-1}$.

Define a function $\varphi: G^{\kappa}_{j}\to\R$ by
\be
 \varphi(x)=\gamma_{j-1}(x)- E_j(x).
\ee
We clearly have that $|\varphi(x)|\leq (\delta_j)^{10}$ In order to prove
the claim, it suffices to prove that there exists an extension
of $\varphi$ to $G^{\kappa}$ satisfying $|\nabla\varphi|\leq \delta_j$.
Let us now recall the conclusions of Lemma~\ref{lem:evanal}.
For $x\in G^{\kappa}_{j}$, we can find a function $g_x: U\to\R$
where $U=B_{(\delta_j)^2}$ such that $\varphi(x)=g_x(x)$ and
$|\nabla g_x|\leq C$. 

\begin{lemma}
 For $y\in G_{j}^{\kappa}\cap U$, we have $g_x(y)=\varphi(y)$.
\end{lemma}

\begin{proof}
 We clearly have that $|g_x(y)|\leq \delta_j^{10} + C (\delta_j)^{2}
 \leq \delta_j^{\frac{3}{2}}$, as $E_j(x)$ is $\delta_j$ simple the
 claim follows.
\end{proof}

The following lemma now guarantees the existence of $\varphi$.

\begin{lemma}
 Let $A\subseteq\R^d$ be a set and $f:A\to\R$ a function
 such that
 \begin{enumerate}
  \item For $x\in A$, $|f(x)|\leq \eps$.
  \item For $x\in A$, there exists $g_x: B_{\delta}(x)\to\R$
   such that $g_x(y)=f(y)$ for $y\in A\cap B_{\delta}(x)$
   and $|g_{x}(y)- f(x)|\leq C |x-y|$.
 \end{enumerate}
 Then there exists $F: \R^d\to\R$ such that $f(x)=F(x)$
 for $x\in A$ and $|\nabla F(x)| \leq C\frac{\eps}{\delta}$.
\end{lemma}

\begin{proof}
 By condition (ii), we can extend $f$ to the $\frac{\delta}{3}$
 neighborhood of $A$ by just setting it equal to the functions $g_{x}$.
 Lets call $f_1$ the extension by $0$ of this function to $\R^d$. 
 Let $\eta:\R^d\to\R$ be a mollifier that is $\int \eta(x) dx =1$,
 $\eta\geq 0$, and $\mathrm{supp}(\eta)\subseteq B_1(0)$.
 We set $\eta_t(x)=\frac{1}{t^d} \eta(x/t)$.
 Finally, we define
 \[
  s(x)=\begin{cases} \frac{\delta}{6},&\dist(x, A)\geq \frac{\delta}{6};\\
   \dist(x,A),&\text{otherwise}. \end{cases}
 \]
 We are now ready to define
 \[
  F(x)=\begin{cases} f(x),&x\in A;\\
   \int_{B_{s(x)}(x)} \eta_{s(x)}(x-y) f_1(y) dy,&x\notin A.\end{cases}
 \]
 First it is clear that $F$ defines a continuous function for
 $x\notin A$. As $f_1$ is continuous and $\eta_{s}\to\delta$
 as $s\to 0$, it follows that $F$ is continuous on $\R^d$.
 To see the estimates on the gradient, we observe that for
 $\dist(x,A)\geq \frac{\delta}{6}$, we have
 \[
 |\nabla F(x)| \leq \|f_1\|_{L^{\infty}(\R^d)} \cdot \|\nabla \eta_{s(x)}\|_{L^1(\R^d)}
   \leq \eps \cdot \frac{6}{\delta} \cdot \|\nabla \eta\|_{L^1(\R^d)}.
 \]
 For $\dist(x,A) < \frac{\delta}{6}$, we have that
 \[
 |\nabla F(x)| \leq \|\nabla f_1\|_{L^{\infty}(B_s(x)(x))} \cdot \|\eta_{s(x)}\|_{L^1(\R^d)}.
 \]
 From this the claim follows.
\end{proof}

In order to prove Proposition~\ref{prop:evfunc} it remains to
prove (vii) that is understand the function $\psi(x)$.
First, it is clear that $\psi_j(x)\to\psi(x)$ in $\ell^2(\Z^d)$ for
$x\in G=\bigcap_{j=1}^{\infty} G_j$. 

\begin{lemma}
 Let $x\in\T^d$, the set of functions 
 \be
  \psi_{\ell}(x;n) = \psi(x-\ell\star\alpha; n+\ell)
 \ee
 where $x-\ell\star\alpha\in G$ form an orthonormal set in $\ell^2(\Z^d)$
 consisting of different eigenfunctions of $\widehat{H}_{\lambda,\alpha,x}$.
\end{lemma}

\begin{proof}
 One can check that $\widehat{H}_{\lambda,\alpha,x}\psi_{\ell}(x)
 =\gamma(x-\ell\star\alpha) \psi_{\ell}(x)$. Hence, the $\psi_{\ell}(x)$
 are all eigenfunctions. Furthermore, they are all different as
 \[
 |\spr{\psi_{\ell}(x)}{\psi_{k}(x)} - \spr{\delta_\ell}{\delta_k}|
   \leq 12 \lambda^{\frac{1}{10}}.
 \]
 If the $\gamma(x-\ell\star\alpha)$ are all different, we are done as the
 eigenfunctions of a self-adjoint operator to different eigenvalues are
 automatically orthonormal.

 Let us now show that $E = \gamma(x-\ell\star\alpha)=\gamma(x-k\star\alpha)$
 cannot happen for $k\neq \ell$ and $x-\ell\star\alpha, x-k\star\alpha \in G$.
 We can assume that $k=0$ and choose $j$ so large that $R_{j}\geq 10 |\ell|$.
 Let
 \[
  \varphi_1(n)= \psi_{j-1}(x;n),\quad \varphi_2(n)=
   \psi_{j-1}(x-\ell\star\alpha;n+ \ell).
 \]
 We have that 
 \[
  \|(\widehat{H}^{\Lambda_{R_j}(n)}_{\lambda,\alpha,x}-E)\varphi_t\|
   \leq \sum_{s\geq j+1} (\delta_s)^{10} + \E^{-\sqrt{R_{j}}} \leq \delta_{j}^{5}
 \]
 for $t=1,2$ and $|\spr{\varphi_1}{\varphi_2}|\leq 12\lambda^{\frac{1}{10}}$.
 This implies that
 \[
  \tr(P_{[E_j(x)-(\delta_j)^2,E_j(x)+(\delta_j)^2]}(\widehat{H}^{\Lambda_{R_j}(0)}
 _{\lambda,\alpha,x})) \geq 2.
 \]
 This is a contradiction finishing the proof.
\end{proof}

%
%

\section{Control on the eigenvalues: Proof of Theorem~\ref{thm:stepA}}
\label{sec:pfstepA}

The first step of the proof will be to prove the following
initial condition.

\begin{proposition}\label{prop:init}
 Let $\eps>0$ then there exists $\delta>0$ such that the following holds.
 Let $R\geq 1$, $x\in\T^d$. Then there exists $\lambda_2 > 0$ such that there exists
 a set $G_1\subseteq [0,1]^d$ of measure $|G_1|\geq 1-\frac{\eps}{2}$
 such that for $\alpha\in G_1$, we have
 \be
  E\text{ is a $\delta$-simple eigenvalue of } 
  \widehat{H}_{\lambda,\alpha,x}^{\Lambda_R(0)}
 \ee
 for $\lambda\in (0,\lambda_2)$ and some $E$ satisfying
 \be
  |E - \sum_{j=1}^{d} 2\cos(2\pi x_j)| \leq \lambda\|f\|_{L^{\infty}(\T^d)}.
 \ee
 Furthermore, the corresponding normalized eigenfunction $\psi$ 
 can be chosen to satisfy  
 \be
  \|\psi -\delta_0\|\leq \frac{2\|f\|_{L^{\infty}(\T^d)}}{\delta}\cdot \lambda.
 \ee
\end{proposition}

We will provide the proof of this proposition in 
Section~\ref{sec:pfinit}. It is essentially a simple test function
construction. We see that $\eps>0$ dictates that $\delta_1$ must
be smaller than $\delta$ thus imposes an additional smallness
condition on $\lambda$, i.e. $\lambda\leq \delta^{10}$.

\bigskip

In order to finish the proof of Theorem~\ref{thm:stepA},
we now need to show that the conclusions for $j-1$ imply
the conclusions of $j$.
In order to accomplish this, we will need to introduce
a bit of notation and review some results from \cite{b07}.
Given $\Lambda\subseteq\Z^d$, $n,m \in \Lambda$, and $E\in\R$,
we introduce the Green's function by
\be
 G^{\Lambda}_{\lambda,\alpha,x}(E;n,m)=\spr{\delta_n}{
  (\widehat{H}^{\Lambda}_{\lambda,\alpha,x}-E)^{-1}\delta_m}.
\ee
In order to quantify the behavior of the Green's function,
we introduce

\begin{definition}
 Let $\gamma>0$, $\tau\in (0,1)$.
 $\Lambda_R(0)$ is called $(\gamma,\tau)$-suitable for
 $\widehat{H}_{\lambda,\alpha,x} - E$ if
 \begin{enumerate}
  \item $\|(\widehat{H}^{\Lambda_R(0)}_{\lambda,\alpha,x}-E)^{-1}\|\leq\E^{R^{\tau}}$.
  \item For $n,m\in\Lambda_R(0)$, $|n-m|\geq R/2$, we have
   \be
   |G^{\Lambda_R(0)}_{\lambda,\alpha,x}(E;n,m)|\leq\E^{-\gamma|n-m|}.
   \ee
 \end{enumerate}
\end{definition}

An adaptation of the argument of \cite{b07} shows that

\begin{theorem}
 Let $\eps >0$. There exist $\lambda_3 > 0$,
 $\gamma>0$, $\sigma,\tau\in (0,1)$ such that for
 $\lambda\in(0,\lambda_3)$, there exists $\mathcal{A}_{1,\lambda}$
 such that for $\alpha\in\mathcal{A}_{1,\lambda}$ and $R\geq 1$,
 we have for $E\in\R$ and $1\leq j\leq d$
 \be
 |\{x_j \in\T:\quad \Lambda_R(0)\text{ is $(\gamma, \tau)$-suitable for }
   \widehat{H}_{\lambda,\alpha,x}-E\}|
    \leq\E^{-R^{\sigma}}
 \ee
 for each fixed choice of $x_1,\dots,x_{j-1}, x_{j+1},\dots,x_d$.
\end{theorem}

\begin{proof}
 This is basically Proposition~2.2. in \cite{b07}. However, Bourgain
 proves this result for $H=\Delta+V$ with $\Delta$ the usual Laplacian.
 The modification to treat operators of the form $\widehat{H}=T+W$,
 where $W$ is a quasi-periodic potential and $T$ a long range hopping
 operator are straightforward. The computations at the beginning
 of Section~\ref{sec:efcontrol} would be helpful to write down a
 proof for the long range case.
\end{proof}

Using semi-algebraic set methods and frequency elimination
one can show

\begin{proposition}\label{prop:freqelim}
 Let $x\in\T^d$, $\gamma>0$, $\tau\in (0,1)$. There exists $C_1 = C_1(d) \geq 1$
 such that for arbitrary $C_2 > C_1$ and $N\geq 1$ large enough, 
 we have that there exists
 $A\subseteq [0,1]^d$ satisfying $|A|\geq 1-\frac{1}{N^2}$
 such that for $\alpha\in G$,
 \be
  N^{C_1} \leq |n|_{\infty} \leq N^{C_2}
 \ee
 and $E\in\R$ satisfying
 \be\label{eq:distspec}
  \dist(E,\sigma(\widehat{H}_{\lambda,\alpha,x}^{\Lambda_{N^{30}}(0)}))\leq \E^{-c N}
 \ee
 we have
 \be
  \Lambda_{N}(n)\text{ is $(\gamma,\tau)$-suitable for }
   \widehat{H}_{\lambda,\alpha,x}-E.
 \ee
\end{proposition}

\begin{proof}
 This can be again achieved as in \cite{b07} see formulas
 (3.5) to (3.26). The only modification necessary is that
 Bourgain works with $N$ instead of $N^{30}$ in \eqref{eq:distspec}.
 The changes required for this are minor.
\end{proof}

With these preparations done, we are now ready to start the
proof of Theorem~\ref{thm:stepA}. We recall that $r=R_{j-1}$
and $R=R_{j}$.
We apply Proposition~\ref{prop:freqelim} with $N^{C_1} = \frac{r}{2}$. So we
let $N = \lfloor (r/2)^{\frac{1}{C_1}}\rfloor$. In order for this
elimination of $\alpha$ only contributing $\frac{\eps}{2}$, we need
to choose $R_1$ so large that
\be\label{eq:condR1}
 \sum_{j=2}^{\infty} \frac{1}{(R_j / 2)^{\frac{2}{C_1}}} \leq \frac{\eps}{2}.
\ee
This is clearly possible. Furthermore, as $C_2$ is arbitrary,
we can choose it such that $(R_j / 2)^{\frac{C_2}{C_1}} \geq R_{j+1} = (R_j)^{10}$.

We will now use the following result about general operators.

\begin{theorem}\label{thm:evexten}
 Let $\gamma>0$, $\tau\in (0,1)$, $1000 \rho \leq r \leq \frac{1}{1000} R$,
 and $\delta\geq\E^{-\frac{\gamma r}{1000}}$.

 Let $E$ be a $\delta$-simple eigenvalue of $H^{\Lambda_r(0)}$ and
 denote by $\psi$ the corresponding normalized eigenfunction.
 Assume for $\frac{r}{2}\leq |n|_{\infty}\leq R$ that
 \be
  \Lambda_{\rho}(n)\text{ is $(\gamma,\tau)$-suitable for }
   H-E.
 \ee
 Then there exists $\ti{E}$ a $\E^{- 300 \gamma \rho}$-simple
 eigenvalue of $H^{\Lambda_R(0)}$ satisfying $|E-\ti{E}|\leq \E^{-\frac{\gamma r}{250}}$.
 Furthermore, there exists a corresponding normalized eigenfunction
 $\varphi$ such that
 \be
  \|\psi-\varphi\|\leq \E^{-\frac{\gamma r}{1000000}}.
 \ee
\end{theorem}

We are now ready for

\begin{proof}[Proof of Theorem~\ref{thm:stepA}]
 We choose $R_1$ by \eqref{eq:condR1}
 and apply Proposition~\ref{prop:init} with $R=R_1$. We now
 see that Theorem~\ref{thm:stepA} holds with possibly imposing
 an additional smallness condition on $\lambda$. We finish the proof
 by the use of induction.

 We eliminate frequencies $\alpha$ using Proposition~\ref{prop:freqelim}
 obtaining a set $G_{j} \subseteq G_{j-1}$. Let now $\alpha\in G_j$.
 Choose $\ell < j$ such that
 \[
  N^3 \leq R_\ell \leq N^{30},\quad N=\lfloor(r/2)^{\frac{1}{C_1}}\rfloor,
   \quad r=R_{j-1}
 \]
 where $N$ is as in Proposition~\ref{prop:freqelim}. We then have that
 \[
  \dist(E_{j-1}(\alpha),\sigma(H^{\Lambda_{R_{\ell}}(0)}_{\lambda,\alpha,x}))
   \leq 2 \delta_{\ell} \leq \E^{-N^{\frac{3}{2}}}.
 \]
 As the corresponding eigenfunction is well localizated, the same
 statement holds for $\sigma(H^{\Lambda_{N^{30}}(0)}_{\lambda,\alpha,x})$.
 Hence, we may conclude from Proposition~\ref{prop:freqelim} that the assumptions
 of Theorem~\ref{thm:evexten} hold. 
 We set $E_j(\alpha)=\ti{E}$, $\psi_j(\alpha)=\varphi$ so that
 \[
 |E_j(\alpha)-E_{j-1}(\alpha)|\leq \E^{-\frac{\gamma}{250} R_{j-1}},
   \quad |\psi_{j}(\alpha)-\psi_{j-1}(\alpha)|
    \leq \E^{-\frac{\gamma}{1000000} R_{j-1}}
 \]
 and
 \[
  E_j(\alpha)\text{ is a $\E^{-300 \gamma (R_{j-1})^{\frac{1}{C_1}}}$-
  simple eigenvalue of }
   \widehat{H}_{\lambda,\alpha,x}^{\Lambda_{R_{j}}(0)}.
 \]
 Theorem~\ref{thm:stepA} now follows by simple arithmetic.
\end{proof}

%
%
%

\section{Proof of the initial condition}
\label{sec:pfinit}

Let $x\in\T^d$, $\alpha\in [0,1]^d$, we consider
\be
 W_{x,\alpha} (n) = \sum_{j=1}^{d} 2\cos(2\pi (x_j+n_j\alpha_j)).
\ee
We have that $W_{x,\alpha}(0)$ only depends on $x\in\T^d$,
whereas all the $W_{x,\alpha}(n)$ for $n\in\Z^d\setminus\{0\}$
are nontrivial functions of $\alpha\in [0,1]^d$. This
immediately implies

\begin{proposition}
 Let $\eps >0$, $R\geq 1$, $x\in\T^d$, and $E=W_{x,\bullet}(0)$.
 There exists $\kappa>0$ and $G\subseteq [0,1]^d$ such
 that $|G|\geq 1-\eps$ and for $n\in\Lambda_R(0)\setminus\{0\}$,
 $\alpha\in G$, we have
 \be
  |W_{x,\alpha}(n) - E|\geq\kappa.
 \ee
\end{proposition}

\begin{proof}
 As $f(y)=\sum_{j=1}^{d}2\cos(2\pi y_j)$ is a real-analytic function.
 There exists $\eta_0 > 0$ and $\beta>0$ such that for each $j=1,\dots,d$
 \[
  G_{\eta} = \{y_j:\quad |f(y)-E|\leq \eta\}
 \]
 has measure $|G_{\eta}|\leq\eta^{\beta}$ for $\eta\in(0,\eta_0)$
 and any choice of $y_1,\dots,y_{j-1},y_{j+1},\dots,y_d$.

 For $n\neq 0$, there is $j=1,\dots,d$ such that $n_j\neq 0$.
 It follows that
 \[
  |W_{x,\alpha}(n) - E|\geq\eta,\quad n\in\Lambda_R(0)\setminus\{0\}
 \]
 for $\alpha$ outside a set of measure $\#\Lambda_R(0) \cdot \eta^{\beta}$.
 Hence, the claim follows for an appropriate choice of $\eta$.
\end{proof}

We thus obtain for $E$ is a $\kappa$-simple eigenvalue
of $W_{x,\alpha}^{\Lambda_R(0)}$ for $\alpha\in G$.
Consider now the operator $\widehat{H}_{\lambda,\alpha,x}
= \lambda T + W_{x,\alpha}$. We clearly have that
\be
 \|(\widehat{H}_{\lambda,\alpha,x}^{\Lambda_{R}(0)} - E)
  \delta_0\| \leq \lambda \|T\|
\ee
and that for $\lambda<\frac{\kappa}{2\|T\|}$
\be
 \tr(P_{[E-\frac{\kappa}{2},E+\frac{\kappa}{2}]}
  \widehat{H}_{\lambda,\alpha,x}^{\Lambda_{R}(0)})) = 1.
\ee
Hence, there exists $E_1$ such that $|E-E_1|\leq \lambda\|T\|<\frac{\kappa}{2}$
and a normalized $\psi\in\ell^2(\Lambda_R(0))$ such that
\be
 \widehat{H}^{\Lambda_R(0)}_{\lambda,\alpha,x} \psi = E_1 \psi.
\ee

\begin{lemma}
 We have
 \be
  \|\psi-\delta_0\| \leq \frac{2 \lambda \|T\|}{\kappa}.
 \ee
\end{lemma}

\begin{proof}
 We have $0 = (\widehat{H}^{\Lambda_R(0)}_{\lambda,\alpha,x} - E_1) \psi(n)
 = (W_{x,\alpha}(n) - E_1) \psi(n) + \lambda (T\psi)(n)$. Thus, we obtain
 for $n\neq 0$ that
 \[
  \kappa \sum_{n\in\Lambda_R(0)\setminus\{0\}} |\psi(n)|^2 \leq \lambda \|T\|.
 \]
 This implies $|\psi(0)|^2=1-\sum_{n\in\Lambda_R(0)\setminus\{0\}} |\psi(n)|^2
 \geq 1-\frac{\lambda\|T\|}{\kappa}$.
\end{proof}

This is all we need to prove Proposition~\ref{prop:init}.

%
%
%

\section{Understanding eigenfunctions}
\label{sec:efcontrol}

In order to begin this section, let me review the assumptions about the
operator $H=\lambda T+W:\ell^2(\Z^d)\to\ell^2(\Z^d)$, where $\lambda>0$
is a small parameter.
We will assume that $T$ is of the form
\be
 T\psi(n)=\sum_{k\in\Z^{d}\setminus\{0\}} t_{n,k} \psi(n+k)
\ee
where we have that $|t_{n,k}| \leq \E^{-\eta |k|_{\infty}}$. This is slightly
more general than what we need as we have $t_{n,k}=\hat{f}(k)$. However,
this greater generality doesn't add any new difficulty.
$W$ is the multiplication operator by a sequence $W\in\ell^{\infty}(\Z^d)$.
For $\Lambda\subseteq\Z^{d}$, we have that
\be
 T^{\Lambda}\psi(n)=\mathop{\sum_{k\in\Z^{d}\setminus\{0\}}}_{n+k\in\Lambda} t_{n,k} \psi(n+k)
\ee
Finally, we recall that for $n,m\in\Lambda$ and $E\in\R$,
the Green's function is defined by
\be
 G^{\Lambda}(E;n,m)=\spr{\delta_n}{(H^{\Lambda}-E)^{-1}\delta_m}.
\ee
Here $H^{\Lambda}$ denotes the restriction of $H$ to $\ell^2(\Lambda)$.
We will now need to prove the following lemma, which allows
us to estimate solutions in terms of the Green's function.

\begin{lemma}
 Let $\Lambda\subseteq\Xi\subseteq\Z^d$ and
 $\psi$ solve $H^{\Xi}\psi=E\psi$. Then
 for $n\in\Lambda$, we have
 \be
  \psi(n)=-\lambda\sum_{m\in\Lambda} G^{\Lambda}(E; n,m)
   \sum_{\ell\in\Xi\setminus\Lambda}t_{m,\ell-m} \psi(\ell).
 \ee
\end{lemma}

\begin{proof}
 We have that $\psi(n)=\spr{\delta_n}{\psi}=\spr{(H^{\Lambda}-E)^{-1}\delta_n}
 {(H^{\Lambda}-E)\psi}$. As $(H^{\Lambda}-E)\psi= \chi_{\Lambda}(H^{\Lambda}-E)\psi$
 and $(H^{\Lambda}-E)\psi=(H^{\Lambda}-E - (H^{\Xi} - E))\psi=\lambda(T^{\Lambda}-T^{\Xi})\psi$,
 we obtain
 \[
  \psi(n)=\lambda\spr{(H^{\Lambda}-E)^{-1} \delta_n}{\chi_{\Lambda}(T^{\Lambda}-T^{\Xi})\psi}.
 \]
 For $m\in\Lambda$, we have that
 \[
  \chi_{\Lambda}(T^{\Lambda}-T^{\Xi})\psi (m) =  -
   \mathop{\sum_{k\in\Z^{d}\setminus\{0\}}}_{m+k\in\Xi\setminus\Lambda} t_{m,k} \psi(m+k) =
    \sum_{\ell\in\Xi\setminus\Lambda} t_{m,\ell-m} \psi(\ell).
 \]
 The claim follows.
\end{proof}

This lemma together with our decay assumption on $t_{n,m}$ implies
\be\label{eq:psiesti1}
|\psi(n)|\leq\lambda\sum_{m\in\Lambda}|G^{\Lambda}(E;n,m)|
  \sum_{\ell\in\Xi\setminus\Lambda}
   \E^{-\eta |\ell-m|_{\infty}} \cdot |\psi(\ell)|.   
\ee
In particular, we obtain

\begin{lemma}\label{lem:psiesti2}
 Let $\gamma\in (0,\frac{\eta}{4})$ and $R\geq 1$ large enough (depending on $\gamma,\eta, d,\tau$).

 Let $\Lambda_R(0)$ be $(\gamma,\tau)$-suitable for $H - E$ and 
 $\psi$ solve $H^{\Xi}\psi=E\psi$ for $\Lambda_R(0)\subseteq \Xi$, $\|\psi\|=1$.
 Then for $|n|\in\Lambda_{\frac{R}{4}}(0)$,
 \be
 |\psi(n)|\leq\lambda\E^{-\frac{\gamma}{2} \dist(n,\partial\Lambda_R(0))} 
   \cdot \max_{m\in\Xi}\left(\E^{-\frac{\eta}{2} \dist(m,\Lambda_R(0))} |\psi(m)|\right).
 \ee
\end{lemma}

\begin{proof}
 Let $\Psi=\max_{m\in\Xi}\left(\E^{-\frac{\eta}{2} \dist(m,\Lambda_R(0))} |\psi(m)|\right)$. 
 We begin by observing that \eqref{eq:psiesti1} implies
 \[
 |\psi(n)|\leq\left(\sum_{m\in\Lambda} |G^{\Lambda}(E;n,m)|
   \sum_{\ell\in\Xi\setminus\Lambda} \E^{-\frac{\eta}{2} |\ell-m|_{\infty}}
    \right) \cdot \Psi.
 \]
 Thus we have $|\psi(n)|\leq\lambda(I_1+I_2) \Psi$,
 where
 \[
  I_1=\sum_{m\in\Lambda_{R}(0)\setminus\Lambda_{R/2}(0)} |G^{\Lambda}(E;n,m)|
   \sum_{\ell\in\Xi\setminus\Lambda_{R}(0)}
    \E^{-\frac{\eta}{2} |\ell-m|_{\infty}}
 \]
 and
 \[
  I_2=\sum_{m\in\Lambda_{R/2}(0)}  |G^{\Lambda}(E;n,m)|
   \sum_{\ell\in\Xi\setminus\Lambda_{R}(0)} 
    \E^{-\frac{\eta}{2} |\ell-m|_{\infty}}.
 \]
 To estimate $I_1$, we have that $|G^{\Lambda}(E;n,m)|\leq \E^{-\gamma |n-m|}$.
 Furthermore, we have that $|\ell-m|_{\infty}\geq (|\ell|_{\infty} -R)
 + (R-|m|_{\infty})$. From this, we conclude
 \begin{align*}
  I_1 &\leq C_1 \cdot\sum_{m\in\Lambda_{R}(0)\setminus\Lambda_{R/2}(0)} \E^{-\gamma |n-m|_{\infty}
   - \frac{\eta}{2} (R - |m|_{\infty})}\\
  &\leq C_2 R^d \E^{-\min(\gamma,\frac{\eta}{2}) \dist(n,\partial\Lambda_{R}(0))}.
 \end{align*}
 To estimate $I_2$, we use that $|G^{\Lambda}(E;n,m)|\leq\E^{\Gamma}$
 and $|\ell-m|\geq\frac{R}{2}+(|\ell|_{\infty}-R)$. From this we conclude
 \[
  I_2\leq C \E^{-\frac{R}{2} \cdot \frac{\eta}{2} + \Gamma}.
 \]
 The claim  follows.
\end{proof}

We will furthermore need the following lemma which allows
us to construct test functions by restricting them.

\begin{lemma}\label{lem:testfunc}
 Let $\|\psi\|=1$ solve $H^{\Xi}\psi=E\psi$ and assume
 \be
  |\psi(n)|\leq \delta,\quad n\in\Lambda_{2R}(0)\setminus\Lambda_R(0)
 \ee
 for $\delta\geq \E^{-\frac{1}{10} \eta R}$.
 Then $\varphi = \chi_{\Lambda_{\frac{3}{2} R}(0)} \psi$ satisfies
 \be
  \|(H^{\Xi} - E) \varphi\| \leq (10R)^{2d} \delta.
 \ee
\end{lemma}

\begin{proof}
 Let $n\in\Xi$, then we have that
 \[
 (H^{\Xi} - E)\varphi(n)=\lambda
  \mathop{\sum_{k\in\Z^{d}\setminus\{0\}}}_{n+k\in\Lambda_{\frac{3}{2}R}(0)}
  t_{n,k} \varphi(n+k)+(W(n)-E)\varphi(n).
 \]
 For $n\in\Lambda_{\frac{3}{2}R}(0)$, this is equal to
 \[
  (H^{\Xi} - E)\varphi(n)=\lambda
   \mathop{\sum_{k\in\Z^{d}\setminus\{0\}}}
   _{n+k\in\Xi\setminus\Lambda_{\frac{3}{2}R}(0)}  t_{n,k} \psi(n+k).
 \]
 Thus, we may estimate
 \begin{align*}
  (H^{\Xi} - E)\varphi(n)&\leq\lambda\left(\delta
   \mathop{\sum_{k\in\Z^{d}\setminus\{0\}}}
   _{n+k\in\Lambda_{2R}(0)\setminus\Lambda_{\frac{3}{2}R}(0)}  t_{n,k}+
     \mathop{\sum_{k\in\Z^{d}\setminus\{0\}}}
     _{n+k\in\Xi\setminus\Lambda_{2 R}(0)}  t_{n,k}\right)\\
       &\leq \lambda\delta(4R)^d+\lambda\sum_{|k|_{\infty}\geq\frac{1}{2} R} \E^{-\eta |k|_{\infty}},
 \end{align*}
 which clearly satisfies what we need.
 For $n\in\Xi\setminus\Lambda_{\frac{3}{2}R}(0)$, one has
 \[
 (H^{\Xi} - E)\varphi(n)=\lambda
  \mathop{\sum_{k\in\Z^{d}\setminus\{0\}}}_{n+k\in\Lambda_{\frac{3}{2}R}(0)}
  t_{n,k} \varphi(n+k).
 \]
 This sum can be estimated as the previous one and the claim follows.
\end{proof}

With these preparations done, we now begin the actual proof
of Theorem~\ref{thm:evexten}.

\begin{lemma}
 Let $1000 \rho < r< R\leq\E^{\rho^{\frac{1}{2}\tau}}$
 and $\rho$ be large enough. Assume that for $\frac{r}{2}\leq |n|_{\infty}\leq R$,
 we have
 \be
  \Lambda_{\rho}(n)\text{ is $(\gamma,\tau)$-suitable for } H-E.
 \ee
 Then we have that
 \be
  \|(H^{\Lambda_{R}(0)\setminus\Lambda_{\frac{r}{2}}(0)}-E)^{-1}\| \leq \E^{5\rho^{\tau}},
   \quad
    \|(H^{\Lambda_{r}(0)\setminus\Lambda_{\frac{r}{2}}(0)}-E)^{-1}\| \leq \E^{5\rho^{\tau}}.
 \ee
\end{lemma}

\begin{proof}
 As $R\geq r \geq 1000 \rho$, it suffices to prove the first claim.
 Assume by contradiction that the claim fails. Then there exists 
 $|\ti{E}-E|\leq\E^{-5 \rho ^{\tau}}$ and $\|\psi\|=1$ solving
 \[
  H^{\Lambda_{R}(0)\setminus\Lambda_{\frac{r}{2}}(0)}\psi=\ti{E}\psi.
 \] 
 Thus there exists $n$ such that $|\psi(n)|\geq \frac{1}{(3R)^d}$.
 By assumption, we have for $\frac{\rho}{2} \leq |m-n|_{\infty} \leq \frac{3\rho}{2}$ that
 $\Lambda_{\rho}(m)$ is $(\gamma,\tau)$-suitable for $H - E$.
 As $(H-\ti{E})^{-1} = (\mathrm{Id}+ (E-\ti{E}) (H-E)^{-1})^{-1} (H-E)^{-1}$,
 we have that
 \[
  \|(H^{\Lambda_{\rho}(m)} -\ti{E})^{-1}\| \leq 2 \E^{\rho^{\tau}}.
 \]
 Now as 
 \begin{align*}
  |G^{\Lambda_{\rho}(m)}(\ti{E};k,\ell)|&\leq |G^{\Lambda_{\rho}(m)}(E;k,\ell)|\\
   &+ |E-\ti{E}|\cdot\|(H^{\Lambda_{\rho}(m)} -\ti{E})^{-1}\|
   \cdot\|(H^{\Lambda_{\rho}(m)} -E)^{-1}\|
 \end{align*}
 we have that $|G^{\Lambda_{\rho}(m)}(\ti{E};k,\ell)|\leq 2\E^{-3\rho^\tau}$
 for $k\in\Lambda_{\frac{\rho}{10}}(0)$, $\ell\in\Lambda_{\rho}(0)\setminus\Lambda_{\frac{\rho}{3}}(0)$.
 Thus the previous lemma implies that
 \[
 |\psi(m)|\leq \E^{-2 \rho^{\tau}}.
 \]
 $\varphi=\psi\chi_{\Lambda_{\rho}(n)}$ satisfies
 \[
  \|(H^{\Lambda_{\rho}(n)}-\ti{E}) \varphi\|\leq (10R)^{2d}\E^{-2 \rho^{\tau}}, \quad
   \|\varphi\| \geq \frac{1}{(3R)^{d}}.
 \]
 This implies $\|(H^{\Lambda_{\rho}(n)}-\ti{E})^{-1}\|> \frac{1}{(3R)^{2d}} \E^{2\rho^{\tau}}$,
 which is a contradiction..
\end{proof}

Let now $\psi$ be the function from the statement
of Theorem~\ref{thm:evexten}. Define the function
\be
 \psi_1(n)=\begin{cases} \psi(n),& n\in\Lambda_{\frac{3}{4}r}(0)(0);\\
  0,&\text{otherwise}.\end{cases}
\ee

\begin{lemma}
 We have that $\|\psi_1\| \geq \frac{1}{2}$.
 Furthermore for $\Lambda_{\frac{3}{4} r}(0)\subseteq \Xi$, we have
 \be
  \|(H^{\Xi}-E) \psi_1\|\leq \E^{-\frac{\gamma r}{200}}.
 \ee
\end{lemma}

\begin{proof}
 In order to estimate $\psi(n)$ for $||n|_{\infty}-\frac{3}{4} r|\leq \frac{1}{8} r$,
 we can iterate Lemma~\ref{lem:psiesti2} at least $\frac{1}{32} \frac{r}{\rho}$
 many times to obtain that for these $n$
 \[
 |\psi(n)|\leq\E^{-\frac{\gamma r}{100}}.
 \]
 Clearly either $\psi_1$ or $\psi_2=\psi-\psi_1$ satisfy that
 $\|\psi_j\|\geq\frac{1}{2}$. By Lemma~\ref{lem:testfunc},
 we obtain that
 \[
  \|(H^{\Xi}-E) \psi_j\|\leq\E^{-\frac{\gamma r}{200}}.
 \]
 Thus $\|\psi_2\|\geq \frac{1}{2}$ would contradict the previous lemma.
\end{proof}

From this lemma, it is clear that
\be
 \sigma(H^{\Lambda_R(0)})\cap [E-\E^{-\frac{\gamma r}{250}}, E+\E^{-\frac{\gamma r}{250}}]
 =\{\ti{E}\}
\ee 
for some $\ti{E}$.
We will now show

\begin{lemma}
 Let $\varphi$ satisfy $H^{\Lambda_{R}(0)} \varphi = \lambda\varphi$
 for $\lambda \in [E-\E^{-300 \gamma \rho},E+\E^{- 300 \gamma \rho}]$.
 Then there exists $|c|=1$ such that
 \be
  \|\varphi-c\psi\|\leq \E^{-\frac{\gamma r}{1000}}.
 \ee
\end{lemma}

\begin{proof}
 For any $n$, $(H^{\Lambda_{\rho}(n)}-\lambda)^{-1} - (H^{\Lambda_{\rho}(n)}-E)^{-1}
 =(E-\lambda)(H^{\Lambda_{\rho}(n)}-\lambda)^{-1} 
 (H^{\Lambda_{\rho}(n)}-E)^{-1}$. Thus, we have that
 \[
  \|(H^{\Lambda_{\rho}(n)}-E)^{-1}\|\leq 2 \E^{\rho^{\tau}}
 \]
 as $|E-\lambda|\cdot \|(H^{\Lambda_{\rho}(n)}-E)^{-1}\|\leq\frac{1}{2}$.
 This implies in particular, that the estimates on the Green's function
 also hold for $\lambda$ up to an neglible factor of $2$.
 Using this, we can show, as in the previous lemma, that 
 $\varphi_1=\varphi\chi_{\Lambda_{\frac{2}{3} r}(0)}$ satisfies
 $\|\varphi_1\|\geq 1-\E^{-\frac{\gamma r}{300}}$ and
 \[
  \|(H^{\Lambda_r(0)}-E)\varphi_1\|\leq\frac{1}{2}\E^{-\frac{\gamma r}{300}}.
 \]
 Letting $\varphi_1=\spr{\varphi_1}{\psi}\psi+\varphi_1^{\perp}$,
 we obtain as $\|(H^{\Lambda_r(0)}-E)\varphi_1^{\perp}\|\geq\delta\|\varphi_{1}^{\perp}\|$
 that
 \[
 |\spr{\varphi_1}{\psi}|\geq \|\varphi_1\|- \frac{10}{\delta} \E^{-\frac{\gamma r}{300}}.
 \]
 The claim now follows by simple arithmetic.
\end{proof}

\begin{proof}[Proof of Theorem~\ref{thm:evexten}]
 The only thing remaining to prove is that
 \[
  \tr(P_{[E-\E^{-300 \gamma \rho},E+\E^{- 300 \gamma \rho}]}(H^{\Lambda_R(0)})) = 1.
 \]
 Assume otherwise, then there would be two orthogonal
 vectors $\varphi_1$ and $\varphi_2$ that satisfy the conclusions
 of the previous lemma. So
 \[
  0=\spr{\varphi_1}{\varphi_2} = c_1 c_2 - \spr{c_1\psi-\varphi_1}{c_2\psi}
  -\spr{\varphi_1}{c_2\psi-\varphi}.
 \]
 As the last 2 terms are $\leq 2 \E^{-\frac{\gamma}{1000} r}$, the claim follows.
\end{proof}

%
%
%

\end{document}